\documentclass{amsart}
\usepackage{setspace}
\usepackage{a4}
\usepackage{amsthm}
\usepackage{latexsym}
\usepackage{amsfonts}
\usepackage{graphicx}
\usepackage{textcomp}
\usepackage{cite}
\usepackage{enumerate}
\usepackage{amssymb}
\usepackage{hyperref}
\usepackage{amsmath}
\usepackage{tikz}
\usepackage[mathscr]{euscript}
\usepackage{mathtools}
\newtheorem{theorem}{Theorem}[section]

\newtheorem{corollary}[theorem] {Corollary}
\newtheorem{definition}[theorem]{Definition}

\title{This is the title}
\usepackage{fancyhdr}

\begin{document}
\hrule\hrule\hrule\hrule\hrule
\vspace{0.3cm}	
\begin{center}
{\bf\large{{Noncommutative Heisenberg-Robertson-Schrodinger  Uncertainty Principles}}}\\
\vspace{0.3cm}
\hrule\hrule\hrule\hrule\hrule
\vspace{0.3cm}
\textbf{K. Mahesh Krishna}\\
School of Mathematics and Natural Sciences\\
Chanakya University Global Campus\\
NH-648, Haraluru Village\\
Devanahalli Taluk, 	Bengaluru  Rural District\\
Karnataka State, 562 110, India\\
Email: kmaheshak@gmail.com\\

Date: \today
\end{center}

\hrule\hrule
\vspace{0.5cm}
\textbf{Abstract}: Let $\mathcal{E}$ be a  Hilbert C*-module  over a unital C*-algebra $\mathcal{A}$. 	Let  $A:	\mathcal{D}(A) \subseteq \mathcal{E} \to  \mathcal{E}$ and $B:	\mathcal{D}(B)\subseteq \mathcal{E}\to  \mathcal{E}$  be  possibly unbounded self-adjoint morphisms. Then for all $x \in \mathcal{D}(AB)\cap  \mathcal{D}(BA)$ with $\langle x, x \rangle =1$, we show that 
\begin{align*}
(1)	\quad \Delta _x(B)^2d_x(A)^2+\Delta _x(A)^2d_x(B)^2\geq \frac{(\langle \{A,B\}x, x \rangle -\{\langle Ax, x \rangle,\langle Bx, x \rangle\})^2-(\langle [A,B]x, x \rangle +[\langle Ax, x \rangle,\langle Bx, x \rangle])^2}{2}
\end{align*}
and 
\begin{align*}
(2)	\quad \Delta _x(A)\Delta _x(B)\geq \frac{\sqrt{\|(\langle \{A,B\}x, x \rangle -\{\langle Ax, x \rangle,\langle Bx, x \rangle\})^2-(\langle [A,B]x, x \rangle +[\langle Ax, x \rangle,\langle Bx, x \rangle])^2\|}}{2}, 
\end{align*}
where 
\begin{align*}
		&\Delta _x(A)\coloneqq \|Ax-\langle Ax, x \rangle x \|, \quad 	d_x(A)\coloneqq \sqrt{\langle Ax, Ax \rangle -\langle Ax, x \rangle^2},\\
		&[A,B] \coloneqq AB-BA, \quad \{A,B\}\coloneqq AB+BA, \\
		& \{\langle Ax, x \rangle,\langle Bx, x \rangle\}\coloneqq \langle Ax, x \rangle\langle Bx, x \rangle +\langle Bx, x \rangle\langle Ax, x \rangle, \\
		& [\langle Ax, x \rangle,\langle Bx, x \rangle]\coloneqq \langle Ax, x \rangle\langle Bx, x \rangle -\langle Bx, x \rangle\langle Ax, x \rangle.
\end{align*}
We call Inequalities (1) and (2) as noncommutative  Heisenberg-Robertson-Schrodinger uncertainty principles. They  reduce to the Heisenberg-Robertson-Schrodinger uncertainty principle  (derived by Schrodinger in 1930) whenever $\mathcal{A}=\mathbb{C}$.

\hrule

\hrule
\section{Introduction}
Let $\mathcal{H}$ be a complex Hilbert space and $A$ be a possibly unbounded self-adjoint operator defined on the domain $\mathcal{D}(A)\subseteq \mathcal{H}$. For $h \in \mathcal{D}(A)$ with $\|h\|=1$, define the \textbf{uncertainty} of $A$ at the point $h$ as 
\begin{align*}
	\Delta _h(A)\coloneqq \|Ah-\langle Ah, h \rangle h \|=\sqrt{\|Ah\|^2-\langle Ah, h \rangle^2}. 
\end{align*}
In 1929, Robertson \cite{ROBERTSON} derived the following mathematical form of the uncertainty principle (also known as uncertainty relation) by  Heisenberg  in 1927 \cite{HEISENBERG}. Recall that for two linear operators $A:	\mathcal{D}(A)\subseteq \mathcal{H}\to  \mathcal{H}$ and $B:\mathcal{D}(B)\subseteq \mathcal{H}\to  \mathcal{H}$, we define the commutator $[A,B] \coloneqq AB-BA$ and the anti-commutator $\{A,B\}\coloneqq AB+BA$.
\begin{theorem} \cite{ROBERTSON, HEISENBERG, VONNEUMANNBOOK, DEBNATHMIKUSINSKI}  (\textbf{Heisenberg-Robertson Uncertainty Principle})
	Let  $A:\mathcal{D}(A)\subseteq \mathcal{H}\to  \mathcal{H}$ and $B:\mathcal{D}(B)\subseteq \mathcal{H}\to  \mathcal{H}$  be self-adjoint operators. Then for all $h \in \mathcal{D}(AB)\cap  \mathcal{D}(BA)$ with $\|h\|=1$, we have 
	\begin{align}\label{HR}
		\frac{1}{2} \left(\Delta _h(A)^2+	\Delta _h(B)^2\right)\geq \frac{1}{4} \left(\Delta _h(A)+	\Delta _h(B)\right)^2 \geq  \Delta _h(A)	\Delta _h(B)   \geq  \frac{1}{2}|\langle [A,B]h, h \rangle |.
	\end{align}
\end{theorem}
In 1930,  Schrodinger improved Inequality (\ref{HR}) \cite{SCHRODINGER}. 
\begin{theorem} \cite{SCHRODINGER, BERTLMANNFRIIS} \label{SE}
	(\textbf{Heisenberg-Robertson-Schrodinger  Uncertainty Principle})
	Let  $A:	\mathcal{D}(A)\subseteq \mathcal{H}\to  \mathcal{H}$ and $B:	\mathcal{D}(B)\subseteq \mathcal{H}\to  \mathcal{H}$  be self-adjoint operators. Then for all $h \in \mathcal{D}(AB)\cap  \mathcal{D}(BA)$ with $\|h\|=1$, we have 
	\begin{align*}
		\Delta _h(A)	\Delta _h(B)    \geq |\langle Ah, Bh \rangle-\langle Ah, h \rangle \langle Bh, h \rangle|
		&=\frac{\sqrt{|\langle \{A,B\}h, h \rangle -2\langle Ah, h \rangle\langle Bh, h \rangle|^2+|\langle [A,B]h, h \rangle |^2}}{2}\\
		&=\frac{\sqrt{(\langle \{A,B\}h, h \rangle -2\langle Ah, h \rangle\langle Bh, h \rangle)^2-\langle [A,B]h, h \rangle ^2}}{2}.
	\end{align*}	
\end{theorem}
We are fundamentally concerned with the question: What is the  noncommutative analogues of Theorem  \ref{SE}? We then naturally look to  the notion of Hilbert C*-modules which  are first introduced by Kaplansky \cite{KAPLANSKY} for modules over commutative C*-algebras and later developed for modules over arbitrary C*-algebras by Paschke  \cite{PASCHKE} and Rieffel \cite{RIEFFEL}. 
\begin{definition}\cite{KAPLANSKY, PASCHKE, RIEFFEL}
	Let $\mathcal{A}$ be a  unital C*-algebra. A left module 	 $\mathcal{E}$  over $\mathcal{A}$ is said to be a (left)  \textbf{Hilbert C*-module} if there exists a  map $ \langle \cdot, \cdot \rangle: \mathcal{E}\times \mathcal{E} \to \mathcal{A}$ such that the following hold. 
	\begin{enumerate}[\upshape(i)]
		\item $\langle x, x \rangle  \geq 0$, $\forall x \in \mathcal{E}$. If $x \in  \mathcal{E}$ satisfies $\langle x, x \rangle  =0 $, then $x=0$.
		\item $\langle x+y, z \rangle  =\langle x, z \rangle+\langle y, z \rangle$, $\forall x,y,z \in \mathcal{E}$.
		\item  $\langle ax, y \rangle  =a\langle x, y \rangle$, $\forall x,y \in \mathcal{E}$, $\forall a \in \mathcal{A}$.
		\item $\langle x, y \rangle=\langle y,x \rangle^*$, $\forall x,y \in \mathcal{E}$.
		\item $\mathcal{E}$ is complete w.r.t. the norm $\|x\|\coloneqq \sqrt{\|\langle x, x \rangle\|}$,  $\forall x \in \mathcal{E}$.
	\end{enumerate}
\end{definition}
We are going to use the following noncommutative Cauchy-Schwarz inequality. 
\begin{theorem}\cite{PASCHKE, RIEFFEL}
Let $\mathcal{E}$ be a  Hilbert C*-module. Then 
\begin{align*}
	\langle x, y \rangle 	\langle y, x \rangle \leq \|	\langle y, y \rangle \|	\langle x, x \rangle , \quad \forall x, y \in \mathcal{E}.
\end{align*}	
\end{theorem}
We derive noncommutative analogues of Theorem \ref{SE} in Theorem \ref{NH}.

\section{Noncommutative Heisenberg-Robertson-Schrodinger  Uncertainty Principles}

Let $\mathcal{E}$ be a  Hilbert C*-module  over a unital C*-algebra $\mathcal{A}$ and $A$ be a possibly unbounded self-adjoint morphism defined on domain $\mathcal{D}(A)\subseteq \mathcal{E}$. For $x \in \mathcal{D}(A)$ with $\langle x, x\rangle =1$, we define the \textbf{noncommutative norm-uncertainty} of $A$ at the point $x$ as 
\begin{align*}
	\Delta _x(A)\coloneqq \|Ax-\langle Ax, x \rangle x \|=\sqrt{\|\langle Ax, Ax \rangle -\langle Ax, x \rangle^2\|}. 
\end{align*}
We  define the \textbf{noncommutative modular-uncertainty} of $A$ at the point $x$ as 
\begin{align*}
	d_x(A)\coloneqq \sqrt{\langle Ax, Ax \rangle -\langle Ax, x \rangle^2}=\sqrt{\langle Ax -\langle Ax, x \rangle x, Ax -\langle Ax, x \rangle x\rangle}.
\end{align*}
Then $d_x(A)\in \mathcal{A}$ and 
\begin{align*}
	\Delta _x(A)=\|	d_x(A)\|.
\end{align*}
Given two elements $a, b \in \mathcal{A}$, we define $[a,b]\coloneqq ab-ba$ and $\{a,b\}\coloneqq ab+ba$. Similarly, for two morphisms $A:	\mathcal{D}(A)\subseteq \mathcal{E}\to  \mathcal{E}$ and $B:	\mathcal{D}(B)\subseteq \mathcal{E}\to  \mathcal{E}$, we define $[A,B] \coloneqq AB-BA$ and $\{A,B\}\coloneqq AB+BA$.
\begin{theorem}	(\textbf{Noncommutative Heisenberg-Robertson-Schrodinger  Uncertainty Principles})\label{NH}
Let $\mathcal{E}$ be a  Hilbert C*-module  over a unital C*-algebra $\mathcal{A}$. 	Let  $A:	\mathcal{D}(A)\subseteq \mathcal{E}\to  \mathcal{E}$ and $B:	\mathcal{D}(B)\subseteq \mathcal{E}\to  \mathcal{E}$  be self-adjoint morphisms. Then for all $x \in \mathcal{D}(AB)\cap  \mathcal{D}(BA)$ with $\langle x, x \rangle =1$, we have 	
	\begin{enumerate}[\upshape(i)]
		\item 
			\begin{align*}
			\Delta _x(B)^2d_x(A)^2&\geq (\langle Ax, Bx \rangle-\langle Ax, x \rangle\langle Bx, x \rangle)(\langle Bx, Ax \rangle-\langle Bx, x \rangle\langle Ax, x \rangle)\\
			&\geq -\frac{(\langle [A,B]x, x\rangle +[\langle Ax, x \rangle,\langle Bx, x \rangle])^2}{4}\\
			&=\frac{(\langle [A,B]x, x \rangle +[\langle Ax, x \rangle,\langle Bx, x \rangle])(\langle x, [A,B]x \rangle +[\langle Ax, x \rangle,\langle Bx, x \rangle])}{4}.
		\end{align*}
	\item 
		\begin{align*}
		\Delta _x(B)d_x(A)&\geq \sqrt{(\langle Ax, Bx \rangle-\langle Ax, x \rangle\langle Bx, x \rangle)(\langle Bx, Ax \rangle-\langle Bx, x \rangle\langle Ax, x \rangle)}\\
		&\geq \frac{\sqrt{(\langle [A,B]x, x \rangle +[\langle Ax, x \rangle,\langle Bx, x \rangle])(\langle x, [A,B]x \rangle +[\langle Ax, x \rangle,\langle Bx, x \rangle])}}{2}.
	\end{align*}
	\item 
		\begin{align*}
			\Delta _x(A)^2d_x(B)^2&\geq (\langle Bx, Ax \rangle-\langle Bx, x \rangle\langle Ax, x \rangle)(\langle Ax, Bx \rangle-\langle Ax, x \rangle\langle Bx, x \rangle)\\
				&\geq -\frac{(\langle [B,A]x, x\rangle +[\langle Bx, x \rangle,\langle Ax, x \rangle])^2}{4}\\
			&=\frac{(\langle [B,A]x, x \rangle +[\langle Bx, x \rangle,\langle Ax, x \rangle])(\langle x, [B,A]x \rangle +[\langle Bx, x \rangle,\langle Ax, x \rangle])}{4}.
	\end{align*}
\item 
	\begin{align*}
	\Delta _x(A)d_x(B)&\geq \sqrt{(\langle Bx, Ax \rangle-\langle Bx, x \rangle\langle Ax, x \rangle)(\langle Ax, Bx \rangle-\langle Ax, x \rangle\langle Bx, x \rangle)}\\
	&\geq \frac{\sqrt{(\langle [B,A]x, x \rangle +[\langle Bx, x \rangle,\langle Ax, x \rangle])(\langle x, [B,A]x \rangle +[\langle Bx, x \rangle,\langle Ax, x \rangle])}}{2}.
\end{align*}
\item 
\begin{align*}
	 \Delta _x(A)\Delta _x(B)&\geq \|\langle Ax, Bx \rangle-\langle Ax, x \rangle\langle Bx, x \rangle\|\\
	 &\geq\frac{\sqrt{\|(\langle [A,B]x, x \rangle +[\langle Ax, x \rangle,\langle Bx, x \rangle])(\langle x, [A,B]x \rangle +[\langle Ax, x \rangle,\langle Bx, x \rangle])\|}}{2}\\
	 &=\frac{\|\langle [A,B]x, x \rangle +[\langle Ax, x \rangle,\langle Bx, x \rangle]\|}{2}.
\end{align*}
\item 
\begin{align*}
	&\Delta _x(B)^2d_x(A)^2+\Delta _x(A)^2d_x(B)^2\geq \frac{(\langle \{A,B\}x, x \rangle -\{\langle Ax, x \rangle,\langle Bx, x \rangle\})^2-(\langle [A,B]x, x \rangle +[\langle Ax, x \rangle,\langle Bx, x \rangle])^2}{2}\\
	&\quad=\frac{(\langle \{A,B\}x, x \rangle -\{\langle Ax, x \rangle,\langle Bx, x \rangle\})^2+(\langle [A,B]x, x \rangle +[\langle Ax, x \rangle,\langle Bx, x \rangle])(\langle x, [A,B]x \rangle +[\langle Ax, x \rangle,\langle Bx, x \rangle])}{2}.
\end{align*}
\item 
\begin{align*}
&\Delta _x(A)\Delta _x(B)\geq \frac{\sqrt{\|(\langle \{A,B\}x, x \rangle -\{\langle Ax, x \rangle,\langle Bx, x \rangle\})^2-(\langle [A,B]x, x \rangle +[\langle Ax, x \rangle,\langle Bx, x \rangle])^2\|}}{2}\\
&\quad = \frac{\sqrt{\|(\langle \{A,B\}x, x \rangle -\{\langle Ax, x \rangle,\langle Bx, x \rangle\})^2+(\langle [A,B]x, x \rangle +[\langle Ax, x \rangle,\langle Bx, x \rangle])(\langle x, [A,B]x \rangle +[\langle Ax, x \rangle,\langle Bx, x \rangle])\|}}{2}.	
\end{align*}
\item If $\mathcal{A}$ is commutative, then 
\begin{align*}
		\Delta _x(B)^2d_x(A)^2+\Delta _x(A)^2d_x(B)^2&\geq \frac{(\langle \{A,B\}x, x \rangle -2\langle Ax, x \rangle\langle Bx, x \rangle)^2-\langle [A,B]x, x \rangle ^2}{2}\\
		&= \frac{(\langle \{A,B\}x, x \rangle -2\langle Ax, x \rangle\langle Bx, x \rangle)^2+\langle [A,B]x, x \rangle \langle x, [A,B]x \rangle }{2}.
\end{align*}
\item If $\mathcal{A}$ is commutative, then
\begin{align*}
	\Delta _x(A)\Delta _x(B)&\geq \frac{\sqrt{\|(\langle \{A,B\}x, x \rangle -2\langle Ax, x \rangle\langle Bx, x \rangle)^2-\langle [A,B]x, x \rangle^2\|}}{2}\\
	&=\frac{\sqrt{\|(\langle \{A,B\}x, x \rangle -2\langle Ax, x \rangle\langle Bx, x \rangle)^2+\langle [A,B]x, x \rangle\langle x, [A,B]x \rangle\|}}{2}.	
\end{align*}
	\end{enumerate}
\end{theorem}
\begin{proof}
	Let $x \in \mathcal{D}(AB)\cap  \mathcal{D}(BA)$ be such that  $\langle x, x \rangle =1$.
	\begin{enumerate}[\upshape(i)]
		\item 	Using noncommutative Cauchy-Schwarz inequality, 
		\begin{align*}
			&(\langle Ax, Bx \rangle-\langle Ax, x \rangle\langle Bx, x \rangle)(\langle Bx, Ax \rangle-\langle Bx, x \rangle\langle Ax, x \rangle)\\
			&=\langle  Ax -\langle Ax, x \rangle x, Bx -\langle Bx, x \rangle x\rangle\langle  Bx -\langle Bx, x \rangle x, Ax -\langle Ax, x \rangle x\rangle\\
			&\leq \|\langle  Bx -\langle Bx, x \rangle x, Bx -\langle Bx, x \rangle x\rangle\|\langle  Ax -\langle Ax, x \rangle x, Ax -\langle Ax, x \rangle x\rangle\\
			&=\|Bx-\langle Bx, x \rangle x\|^2(\langle Ax, Ax \rangle -\langle Ax, x \rangle^2)=	\Delta _x(B)^2d_x(A)^2.
		\end{align*}
	We now note that 
	\begin{align*}
			&(\langle Ax, Bx \rangle-\langle Ax, x \rangle\langle Bx, x \rangle)(\langle Bx, Ax \rangle-\langle Bx, x \rangle\langle Ax, x \rangle)\\
			&\geq (\text{Im}(\langle Ax, Bx \rangle-\langle Ax, x \rangle\langle Bx, x \rangle))^2\\
			&=\left(\frac{\langle Ax, Bx \rangle-\langle Ax, x \rangle\langle Bx, x \rangle-\langle Bx, Ax \rangle-\langle Bx, x \rangle\langle Ax, x \rangle}{2i}\right)^2\\
			&=-\frac{(\langle [A,B]x, x\rangle +[\langle Ax, x \rangle,\langle Bx, x \rangle])^2}{4}.
	\end{align*}
Now we see that 
\begin{align*}
	-(\langle [A,B]x, x \rangle +[\langle Ax, x \rangle,\langle Bx, x \rangle])^2=(\langle [A,B]x, x \rangle +[\langle Ax, x \rangle,\langle Bx, x \rangle])(\langle x, [A,B]x \rangle +[\langle Ax, x \rangle,\langle Bx, x \rangle]).
\end{align*}
	\item Follows from (i) by noting that square root respects the order of positive elements.
	\item Similar to (i).
	\item Follows from (iii).
	\item We get by taking norm on (i) (norm respects the order of positive elements).
	\item  Define 
	\begin{align*}
		z\coloneqq Ax -\langle Ax, x \rangle x, \quad w\coloneqq Bx -\langle Bx, x \rangle x.
	\end{align*}
	Then 
	\begin{align*}
		\langle z, w\rangle+ \langle w, z\rangle&=\langle  Ax -\langle Ax, x \rangle x, Bx -\langle Bx, x \rangle x\rangle+\langle  Bx -\langle Bx, x \rangle x, Ax -\langle Ax, x \rangle x\rangle\\
		&=\langle BAx, x \rangle-\langle Ax, x \rangle\langle Bx, x \rangle+\langle ABx, x \rangle-\langle Bx, x \rangle\langle Ax, x \rangle\\
		&=\langle \{A,B\}x, x \rangle -\{\langle Ax, x \rangle,\langle Bx, x \rangle\}
	\end{align*}
	and 
	\begin{align*}
		\langle z, w\rangle- \langle w, z\rangle&=\langle  Ax -\langle Ax, x \rangle x, Bx -\langle Bx, x \rangle x\rangle-\langle  Bx -\langle Bx, x \rangle x, Ax -\langle Ax, x \rangle x\rangle\\
		&=\langle BAx, x \rangle-\langle Ax, x \rangle\langle Bx, x \rangle-\langle ABx, x \rangle+\langle Bx, x \rangle\langle Ax, x \rangle\\
		&=-\langle [A,B]x, x \rangle -[\langle Ax, x \rangle,\langle Bx, x \rangle].
	\end{align*}
Now  adding (i) and (iii) gives
\begin{align*}
	&\Delta _x(B)^2d_x(A)^2+\Delta _x(A)^2d_x(B)^2\geq \langle z, w\rangle \langle w, z\rangle+\langle w, z\rangle \langle z, w\rangle\\
	&=\frac{(\langle z, w\rangle+ \langle w, z\rangle)^2-(\langle z, w\rangle- \langle w, z\rangle)^2}{2}\\
	&=\frac{(\langle \{A,B\}x, x \rangle -\{\langle Ax, x \rangle,\langle Bx, x \rangle\})^2-(\langle [A,B]x, x \rangle +[\langle Ax, x \rangle,\langle Bx, x \rangle])^2}{2}.
\end{align*}
\item We get by taking  norm on (vi).
\item Follows from (vi). Note that 
\begin{align*}
	-\langle [A,B]x, x \rangle ^2=\langle [A,B]x, x \rangle \langle x, [A,B]x \rangle.
\end{align*}
\item Follows from (vii).
	\end{enumerate}
\end{proof}
\begin{corollary}
(\textbf{Noncommutative Heisenberg-Robertson  Uncertainty Principles})
Let $\mathcal{E}$ be a  Hilbert C*-module  over a unital C*-algebra $\mathcal{A}$. 	Let  $A:	\mathcal{D}(A)\subseteq \mathcal{E}\to  \mathcal{E}$ and $B:	\mathcal{D}(B)\subseteq \mathcal{E}\to  \mathcal{E}$  be self-adjoint morphisms. Then for all $x \in \mathcal{D}(AB)\cap  \mathcal{D}(BA)$ with $\langle x, x \rangle =1$, we have 
\begin{enumerate}[\upshape(i)]
	\item 
	\begin{align*}
			\Delta _x(B)^2d_x(A)^2+\Delta _x(A)^2d_x(B)^2\geq\frac{(\langle [A,B]x, x \rangle +[\langle Ax, x \rangle,\langle Bx, x \rangle])(\langle x, [A,B]x \rangle +[\langle Ax, x \rangle,\langle Bx, x \rangle])}{2}.
	\end{align*}
	\item 
	\begin{align*}
		\Delta _x(A)\Delta _x(B)&\geq \frac{\|\langle [A,B]x, x \rangle +[\langle Ax, x \rangle,\langle Bx, x \rangle]\|}{2}.
	\end{align*}
\item If $\mathcal{A}$ is commutative, then 
\begin{align*}
	\Delta _x(B)^2d_x(A)^2+\Delta _x(A)^2d_x(B)^2\geq  \frac{\langle [A,B]x, x \rangle \langle x, [A,B]x \rangle }{2}.
\end{align*}
\item If $\mathcal{A}$ is commutative, then
\begin{align*}
	\Delta _x(A)\Delta _x(B)\geq	 \frac{\|\langle [A,B]x, x \rangle\|}{2}.
\end{align*}
\end{enumerate}
\end{corollary}

 \bibliographystyle{plain}
 \bibliography{reference.bib}

\begin{thebibliography}{1}

\bibitem{BERTLMANNFRIIS}
Reinhold~A Bertlmann and Nicolai Friis.
\newblock {\em Modern Quantum Theory: From Quantum Mechanics to Entanglement
  and Quantum Information}.
\newblock Oxford University Press, 2023.

\bibitem{DEBNATHMIKUSINSKI}
Lokenath Debnath and Piotr Mikusi\'{n}ski.
\newblock {\em Introduction to {H}ilbert spaces with applications}.
\newblock Academic Press, Inc., San Diego, CA, 1999.

\bibitem{HEISENBERG}
W.~Heisenberg.
\newblock The physical content of quantum kinematics and mechanics.
\newblock In {\em Quantum Theory and Measurement}, pages 62--84. Princeton
  University Press, Princeton, NJ, 1983.

\bibitem{KAPLANSKY}
Irving Kaplansky.
\newblock Modules over operator algebras.
\newblock {\em Amer. J. Math.}, 75:839--858, 1953.

\bibitem{PASCHKE}
William~L. Paschke.
\newblock Inner product modules over {$B^{\ast} $}-algebras.
\newblock {\em Trans. Amer. Math. Soc.}, 182:443--468, 1973.

\bibitem{RIEFFEL}
Marc~A. Rieffel.
\newblock Induced representations of {$C^{\ast} $}-algebras.
\newblock {\em Advances in Math.}, 13:176--257, 1974.

\bibitem{ROBERTSON}
H.~P. Robertson.
\newblock The uncertainty principle.
\newblock {\em Phys. Rev.}, 34(1):163--164, 1929.

\bibitem{SCHRODINGER}
E.~Schr\"{o}dinger.
\newblock About {H}eisenberg uncertainty relation.
\newblock {\em Bulgar. J. Phys.}, 26(5-6):193--203, 1999.

\bibitem{VONNEUMANNBOOK}
John von Neumann.
\newblock {\em Mathematical foundations of quantum mechanics}.
\newblock Princeton University Press, Princeton, NJ, 2018.

\end{thebibliography}

\end{document}